\documentclass[reqno,10pt]{amsart}
\usepackage{amssymb,amsthm,colordvi}

\usepackage{color,epsf,graphicx,graphics}

\usepackage{latexsym,amsfonts,amssymb,amsmath,amsthm,calc,color,colordvi,epsf,graphicx,graphics,float,upgreek}

\usepackage[ansinew]{inputenc}

\numberwithin{equation}{section}

\newtheorem{defi}{Definition}[section]
\newtheorem{thm}[defi]{Theorem}

\newcommand{\bpr}{\begin{proof}[Proof]}  
\newcommand{\epr}{\end{proof}}
\newcommand{\beq}{\begin{equation}}
\newcommand{\eeq}{\end{equation}}
\newcommand{\bce}{\begin{center}}
\newcommand{\ece}{\end{center}}
\newcommand{\be}{\begin{enumerate}}  
\newcommand{\ee}{\end{enumerate}}
\newcommand{\bc}{\begin{center}}
\newcommand{\ec}{\end{center}}

\renewcommand{\Red}{\Black}

\def\R{\mathbb R}
\def\C{\mathbb C}

\def\N{\mathbb N}

\def\EE{\mathbb E}
\def\FF{\mathbb F}

\def\cA{\mathcal A}
\def\cB{\mathcal B}

\def\cD{\mathcal D}

\def\cE{\mathcal E}

\def\cMH{\mathcal{MH}}

\def\cSM{\mathcal{SM}}

\def\bea{\begin{eqnarray}}
\def\eea{\end{eqnarray}}
\def\beas{\begin{eqnarray*}}
\def\eeas{\end{eqnarray*}}
\newcommand{\ve}{\varepsilon}

\begin{document}

\title[The Verigin Problem]
{ The Verigin Problem\\ with and without Phase Transition}

\author{Jan Pr\"uss}
\address{Martin-Luther-Universit\"at Halle-Witten\-berg\\
         Institut f\"ur Mathematik \\
         Theodor-Lieser-Strasse 5\\
         D-06120 Halle, Germany}
\email{jan.pruess@mathematik.uni-halle.de}

\author{Gieri Simonett}
\address{Department of Mathematics\\
        Vanderbilt University\\
        Nashville, Tennessee\\
        USA}
\email{gieri.simonett@vanderbilt.edu}

\thanks{The research of G.S.\ was partially
supported by the NSF Grant DMS-1265579.}

\subjclass[2010]{35Q35, 76D27, 76E17, 35R37, 35K59}
\keywords{Two-phase flows, phase transition, Darcy's law, Forchheimer's law, available energy, quasilinear parabolic evolution equations, maximal regularity, generalized principle of linearized stability, convergence to equilibria}

\begin{abstract}
Isothermal compressible two-phase flows with and without phase transition are modeled, employing Darcy's and/or Forchheimer's law for the velocity field. It is shown that the resulting systems are thermodynamically consistent in the sense that the available energy is a strict Lyapunov functional. In both cases, the equilibria are identified and their thermodynamical stability is investigated by means of a variational approach. It is shown that the problems are well-posed in an $L_p$-setting and generate local semiflows in the proper state manifolds. It is further shown that a non-degenerate equilibrium is dynamically stable in the natural state manifold if and only if it is thermodynamically stable. Finally, it is shown that a solution which does not develop singularities exists globally and  converges to an equilibrium in the state manifold.
\end{abstract}

\maketitle

\section{Introduction}

\noindent
The Verigin problem concerns compressible two-phase potential flows driven by surface tension. 
It is the compressible analogue to the Muskat problem in which the phases are incompressible. 
In contrast to the Muskat problem, there is only scarce work on the Verigin problem. We only know of the
papers \cite{BiSo00, Fro99, Fro03, Rad04, Tao97, TaYi96, Xu97}, which address local existence in some special cases, mostly excluding surface tension, which is physically questionable.
None of these papers deals with thermodynamical consistency, equilibria, stability questions, and large time behaviour of solutions. Also, there are no results at all on the Verigin problem with phase transition.

It is the aim of this paper to close these gaps. 
We shall develop a fairly complete dynamical theory for the Verigin problem with and without phase transition.
This includes local well-posedness, thermodynamical consistency,  identification of the equilibria, discussion of their stability, the local semiflows on the proper state manifolds, as well as convergence to equilibrium  of solutions which do not develop singularities in a sense to be specified. To a large extent we will follow the strategy and employ the tools of the monograph Pr\"uss and Simonett \cite{PrSi16}.

In Section 2 we derive the model for the Verigin problem with and without phase transition,
following  the arguments of~\cite[Chapter 1]{PrSi16}.
In Sections 3 and~5 we discuss the thermodynamical properties of the model and analyze the stability of  equilbria,
obtaining novel results   not contained in~\cite{PrSi16}.
In Section 4 we  derive the linearization of the Verigin problem and analyze the main symbol of the linearized problem.
Here we can take advantage of the results in Sections~6.6 and~6.7 of~\cite{PrSi16}
which deal with solvability of the linearized Stefan and Verigin problem, respectively.
Well-posedness of the (nonlinear) Verigin problem  in Section~4 is new.
Lastly, in Section~6 we discuss the global behavior of solutions, following the strategy laid out
in~\cite[Section 11.4]{PrSi16}. 

To fix some notation, in this paper $\Omega\subset\R^n$ denotes a bounded domain with outer boundary {$\partial\Omega\in C^2$}, 
{$\varrho$} the density, {$u$} the velocity, and {$\pi$} the pressure field. The domain {$\Omega$} consists of two parts, $\Omega_1$ 
is the so-called {\em disperse phase}
and {$\Omega_2$} is the {\em continuous phase}, {$\Gamma =\partial\Omega_1\subset \Omega$} denotes the {\em interface}. In particular, we assume no {\em boundary contact}. The {outer normal} of {$\Omega_1$} will be denoted by {$\nu_\Gamma$}, the corresponding 
{\em normal velocity} of {$\Gamma$} by {$V_\Gamma$}, and the {\em normal jump} of a quantity {$v$} across {$\Gamma$} by {$[\![v]\!]:= v_2-v_1$}. A typical initial geometry is depicted in Figure 1.

The free energies in the phases will be given functions {$\psi(\varrho)$} which may depend on the phase.
The relation between pressure and density in each phase is given by {Maxwell's law}, which reads
\begin{equation} \pi(\varrho)=\varrho^2\psi^\prime(\varrho),\quad \varrho>0.\end{equation}
We assume that this function is strictly increasing, hence we may invert it to obtain the so-called {\em equation of state} $\varrho=\varrho(\pi)$.

The converse statement is also true. Given an equation of state $\varrho =\varrho(\pi)$ with $\varrho$ strictly increasing, inverting this relation we find $\pi=\phi(\varrho)$ and $\psi(\varrho)$ can then be found from $\psi^\prime(\varrho)= \pi(\varrho)/\varrho^2$, up to a constant. However, in this paper we consider the free energies to be given.

As an example, we consider an {\em ideal gas}, where the equation of state reads $\pi(\varrho)=c\varrho$, with some constant $c>0$. Then we obtain
$$ \psi(\varrho) = c\log(\varrho)  +d,$$
where $d=\psi(1)$ denotes another constant. Another common equation of state reads $\pi(\varrho) =c \varrho^r$, with constants $c>0$, $r>0$, $r\neq1$. In this case we have
$$ \psi(\varrho)= \frac{c}{r-1} \varrho^{r-1} +d.$$

\begin{figure}
\centering
\includegraphics[width=7cm, height=5cm]{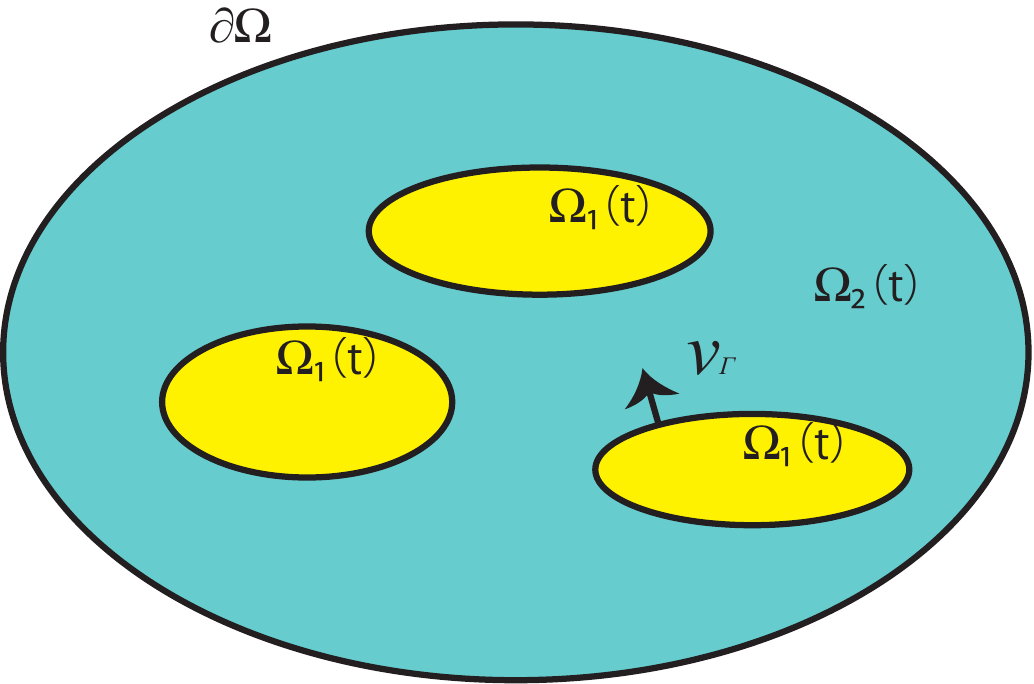} 
\caption{A typical geometry}
\end{figure}
\section{Modeling}

In the sequel we briefly explain the model; cf.\ the monograph Pr\"uss and Simonett \cite{PrSi16}, Chapter 1, for more details.

\bigskip

\noindent
{\bf 2.1 Balance of mass}

\noindent
We assume that there is no surface mass.
 Then balance of mass becomes
\begin{equation}\label{1.1}
\begin{aligned}
\partial_t \varrho+ {\rm div}\, (\varrho u)&=0  &&\mbox{in}\;\; \Omega\setminus\Gamma(t),\\
\mbox{}[\![\varrho(u\cdot\nu_\Gamma-V_\Gamma)]\!] &=0   &&\mbox{on}\;\; \Gamma(t).
\end{aligned}
\end{equation}
At the outer boundary {$\partial\Omega$} we assume {$u\cdot\nu=0$}.

\medskip

\noindent
These equations imply in particular {\em conservation of total mass}
$$ \frac{d}{dt}{\sf M}(t)=0, \quad {\sf M}(t)=\int_\Omega \varrho(t,x)\,dx={\sf M}(0)=:{\sf M}^0.$$
We define the {\em interfacial mass flux} (phase flux for short) by means of
$$ j_\Gamma:= \varrho(u\cdot\nu_\Gamma-V_\Gamma), \quad \mbox{{which means}}\quad {[\![\frac{1}{\varrho}]\!] j_\Gamma= [\![u\cdot \nu_\Gamma]\!]}.$$
Note that $j_\Gamma$ is well-defined by \eqref{1.1}.
We consider two cases.
\begin{itemize}
\item[{\bf (i)}] {\em No phase transition} means {$j_\Gamma\equiv0$}. Then
\begin{equation}\label{1.2}
V_\Gamma:=u\cdot\nu_\Gamma \quad \mbox{ and }\quad  [\![u\cdot\nu_\Gamma]\!]=0,
\end{equation}
i.e.,\ the interface is advected with the flow.
\vspace{2mm}
\item[{\bf (ii)}] {\em Phase transition} means
{$j_\Gamma\not\equiv0$}.  Then
$$V_\Gamma = u\cdot\nu_\Gamma-j_\Gamma/\varrho,$$
which implies
\begin{equation}\label{1.3}
 [\![\varrho]\!] V_\Gamma = [\![\varrho u\cdot\nu_\Gamma]\!]\quad  \mbox{and} \quad [\![1/\varrho]\!]j_\Gamma= [\![u\cdot\nu_\Gamma]\!].
 \end{equation}
Due to the additional variable {$j_\Gamma$}, in this case one more equation on the interface will be needed.
\end{itemize}
There are two more cases.
\begin{itemize}
\item[{\bf (a)}] {\em  The incompresible case}: {\bf Muskat Problems}.\\
Here the density {$\varrho>0$} is assumed to be constant in each phase.
\vspace{2mm}
\item[{\bf (b)}] {\em The compressible case}: {\bf Verigin Problems}.\\
Here the density {$\varrho=\varrho(\pi)>0$} depends on the pressure and satisfies {$\varrho^\prime(\pi)>0$} for all relevant {$\pi\in\R$}.
\end{itemize}

\medskip

\noindent
In this paper we concentrate on the Verigin problems. The Muskat problems reduce to nonlocal geometric evolution  equations which are studied in the monograph  Pr\"uss and Simonett \cite{PrSi16}, Chapter 12; see also Pr\"uss and Simonett \cite{PrSi16a} and the references given there.

\bigskip
\goodbreak
\noindent
{\bf 2.2 Modeling the velocity}

\noindent
The velocity {$u$} is modeled as a potential flow, following {\em Darcy's law}. This means
\begin{equation}\label{2.1} u = -k\nabla\pi,\end{equation}
where $k=k(\pi)>0$ is called {\em permeability}. Note that the function {$k$} depends on the phase. A variant of this is 
{\em Forchheimer's law} which reads
\begin{equation}\label{2.2} g(|u|)u= -l \nabla\pi,\end{equation}
where $l=l(\pi)>0$,
and {$g>0$} is such that the function {$s\mapsto sg(s)$} is strictly increasing. Solving this equation for {$u$} leads to
\begin{equation}\label{2.3} u = -k(\pi,|\nabla\pi|^2)\nabla \pi,\end{equation}
with {$ k(\pi,s)>0$} and {$k(\pi,s)+2s\partial_2 k(\pi,s)>0,$} for all $\pi\in\R$ and $s\ge 0$. 

Conservation of mass then yields the quasilinear diffusion equation
\begin{equation}\label{2.4} 
\varrho^\prime(\pi) \partial_t \pi - {\rm div}(\varrho(\pi)k(\pi,|\nabla\pi|^2)\nabla\pi)=0\quad \mbox{in} \;\; \Omega\setminus\Gamma(t),\end{equation}
and the boundary condition $u\cdot\nu=0$ becomes the Neumann condition
{$ \partial_\nu\pi =0$} on the outer boundary {$\partial\Omega$}.

On the interface, the driving force will be {\em surface tension}, which means
\begin{equation}\label{2.5} [\![\pi]\!] =\sigma H_\Gamma,\end{equation}
where {$H_\Gamma$} denotes the {\em mean curvature} of the interface, and {$\sigma>0$} the (constant) coefficient of surface tension.

Next we have to distinguish the cases {\bf (i)} and {\bf (ii)}.

\medskip

\noindent
{\bf (i)} Here there is no phase transition which means {$j_\Gamma=0$}, hence we obtain by  \eqref{1.2}
$$ 0=[\![u\cdot\nu_\Gamma]\!] = -[\![k(\pi,|\nabla\pi|^2)\partial_\nu \pi]\!]\quad \mbox{on} \;\; \Gamma,$$
and
$$ V_\Gamma = u\cdot\nu_\Gamma = -k(\pi,|\nabla\pi|^2)\partial_\nu\pi\quad \mbox{on} \;\; \Gamma.$$
In this case the mass is  preserved, even in each component of the phases!

\medskip
\noindent
{\bf (ii)} If phase transition is present, then  we obtain by \eqref{1.3}
$$ [\![\varrho(\pi)]\!]V_\Gamma = [\![\varrho(\pi)u\cdot\nu_\Gamma]\!] = -[\![\varrho(\pi)k(\pi,|\nabla\pi|^2)\partial_\nu\pi]\!]\quad \mbox{ on } \; \Gamma.$$
Due to additional variable $j_\Gamma$,  we have to add  another condition on the boundary, which will be the (reduced) {Gibbs-Thomson law}
$$ [\![\psi(\varrho) +\varrho\psi^\prime(\varrho)]\!]=0 \quad \mbox{ on } \; \Gamma.$$
In this case only the total mass is conserved.
Note that here the {free energy} {$\psi(\varrho)$} shows up explicitly, in contrast to the case without phase transition..

Summarizing, we have the following two problems

\bigskip
\goodbreak
\noindent
{\bf 2.3 The Verigin problem without phase transition}

\noindent
The resulting problem becomes
\begin{equation}
\label{Ve-without}
\begin{aligned}
\varrho^\prime(\pi) \partial_t \pi - {\rm div}(\varrho(\pi)k(\pi,|\nabla\pi|^2)\nabla\pi)&=0 &&\mbox{in} \;\; \Omega\setminus\Gamma(t),\\
\partial_\nu\pi &=0  &&\mbox{on} \;\; \partial\Omega,\\
\mbox{} [\![\pi]\!] &=\sigma H_\Gamma && \mbox{on} \:\; \Gamma(t),\\
\mbox{} [\![k(\pi,|\nabla\pi|^2)\partial_\nu \pi]\!]&=0  &&\mbox{on} \; \Gamma(t),\\
V_\Gamma +k(\pi,|\nabla\pi|^2)\partial_\nu\pi& =0  &&\mbox{on} \; \Gamma(t),\\
  \Gamma(0)=\Gamma_0,\quad\pi(0)&=\pi_0 && \mbox{in} \;\; \Omega.
\end{aligned}
\end{equation}
In the sequel, we assume
$$ \varrho\in C^2(\R_+),\quad \varrho(p), \varrho^\prime(p)>0 \quad \mbox{for all}\;\; p\in\R,$$
and
$$ k\in C^2(\R\times\R_+),\quad k(p,s),\; k(p,s)+2s\partial_2k(p,s)>0 \quad \mbox{for all} \;\; p\in\R,\; s\ge 0.$$

\bigskip
\noindent
{\bf 2.4 The Verigin problem with phase transition}

\noindent
This problem reads as follows.
\begin{equation}
\label{Ve-with}
\begin{aligned}
\varrho^\prime(\pi) \partial_t \pi - {\rm div}(\varrho(\pi)k(\pi,|\nabla\pi|^2)\nabla\pi)&=0 &&\mbox{in} \;\; \Omega\setminus\Gamma(t),\\
\partial_\nu\pi &=0 &&\mbox{on} \;\; \partial\Omega,\\
 \mbox{}[\![\pi]\!] &=\sigma H_\Gamma &&\mbox{on} \;\; \Gamma(t),\\
\mbox{}[\![\psi(\varrho) +\varrho\psi^\prime(\varrho)]\!]&=0  &&\mbox{on} \;\; \Gamma(t),\\
\mbox{}[\![\varrho(\pi)]\!] V_\Gamma +[\![\varrho(\pi)k(\pi,|\nabla\pi|^2)\partial_\nu\pi]\!]&=0 &&\mbox{on} \;\; \Gamma(t),\\
  \Gamma(0)=\Gamma_0,\quad\pi(0)&=\pi_0 && \mbox{in} \;\; \Omega.
\end{aligned}
\end{equation}
It should be observed that, besides  the previous assumptions on {$k$} and {$\varrho$}, this problem will only be well-posed if {$[\![\varrho]\!]\neq0$}, in contrast to the case without phase transition.

\section{Thermodynamical Properties of the Models}

In this section some physical properties of the models are discussed. We first introduce the available energy ${\sf E}_a$.

\bigskip

\noindent
{\bf 3.1 The available energy and equilibria}

\noindent
The available energy ${\sf E}_a $ is given by
$$ {\sf E}_a ={\sf E}_a(\pi,\Gamma) = \int_\Omega \varrho\psi \,dx + \sigma {\rm mes}(\Gamma);$$
it is the sum of free and surface energy. A short computation yields
$$ \frac{d}{dt} {\sf E}_a(t) = -\int_\Omega k(\pi,|\nabla\pi|^2)|\nabla\pi|^2\,dx\leq 0,$$
hence {${\sf E}_a$} is a {Lyapunov functional} for both Verigin problems. Note that in case {\bf (ii)} there is no energy dissipation on the interface, due to the Gibbs-Thomson relation.

To see that ${\sf E}_a$  is even a {strict Lyapunov functional}, suppose  {$\frac{d}{dt} {\sf E}_a(t)=0$} at some time {$t$}. As {$k>0$} this implies
{$\nabla\pi=0$},
hence {$\pi$} is constant in the components of the phases, and moreover with {$\varrho^\prime>0$} this yields {$\partial_t \pi=0$} as well as
{$V_\Gamma =0$} in case {\bf (i)}, and also in case {\bf (ii)} if {$[\![\varrho]\!]\neq0$}. Thus we are at an equilibrium, which proves that the available energy is even a {strict Lyapunov functional}.

Via the interface condition {$[\![\pi]\!] =\sigma H_\Gamma$} this further shows  that  {$H_\Gamma$} is constant on the components of the interface.
Therefore, the (non-degenerate) equilibria are constant pressures in the components of the phases, and {$\Gamma$} is a disjoint union of finitely many disjoint spheres
{$\Gamma_j=S_{R_j}(x_j)$}, say {$j=1,\dots,m$},
such that
$$ [\![\pi]\!] = -\frac{(n-1)\sigma}{R_j}\quad \mbox{ on } \; \Gamma_j,\quad j=1,\ldots,m.$$
The set of {non-degenerate equilibria} is denoted by {$\cE$} in the sequel. \\
Now we have to distinguish the cases.

\medskip

\noindent
{\bf (i)} {\em Without phase transition.}\\
In this case there are no further restrictions, hence the manifold of equilibria has dimension {${\rm dim}\,\cE= m(n+1)+1$}. Prescribing the masses of the components of the phases, this yields $(m+1)$ conditions, reducing the degrees of freedom to {$mn$}. We emphasize that in this case
the radii {$R_j>0$} of the spheres  are arbitrary and the continuous phase {$\Omega_2$} need not be connected.

\medskip

\noindent
{\bf (ii)} {\em With phase transition.}\\
Here we have the additional interface condition {$[\![\psi(\varrho)+\varrho\psi^\prime(\varrho)]\!]=0$}. As the functions
\begin{equation}
\label{varphi}
 \varphi(\varrho):=\psi(\varrho)+\varrho\psi^\prime(\varrho)\quad \mbox{satisfy}\quad \varphi^\prime(\varrho)= \pi^\prime(\varrho)/\varrho>0,
\end{equation}
this shows that {$\varrho_2$} uniquely determines {$\varrho_1$} and vice versa, hence the same is valid for {$\pi_i$}. Therefore, the densities and pressures are constant even throughout the phases, and so the spheres all have the same radius. Consequently, {$\Omega_2$} is connected, and  the dimension of {$\cE$} in this case is
{${\rm dim}\,\cE= mn+1$}; conservation of mass reduces it by one.

\bigskip

\noindent
{\bf 3.2 The variational approach: first variation}

\noindent
Consider the functional {${\sf E}_a(\pi,\Gamma)$}, i.e.,\ the available energy, with constraints

\medskip

\noindent
{\bf (i) } {\em Without phase transition}.\\
$$ {\sf M}_{ij}(\varrho,\Gamma) =\int_{\Omega_{ij}} \varrho \,dx = {\sf M}_{ij}(0)=: {\sf M}_{ij}^0.$$
This encodes conservation of mass of the components {$\Omega_{ij}$} of the phases {$\Omega_i$}, for {$i=1,2$}.

\medskip

\noindent
{\bf (ii)} {\em With phase transition}.
$$ {\sf M}(\varrho,\Gamma) = \int_\Omega \varrho \,dx = {\sf M}(0)=:{\sf M^0},$$
which means  conservation of total mass.

\medskip

\noindent
The method of Lagrange multipliers at a critical point {$e_*:=(\pi_*,\Gamma_*)$} with these constraints yields
$${\sf E}_a^\prime(e_*) + \sum_{ij}\mu_{ij} {\sf M}_{ij}^\prime(e_*) =0,\quad \mbox{resp.} \quad {\sf E}_a^\prime(e_*)
+  \mu{\sf M}^\prime(e_*) =0,$$ for some constants {$\mu_{ij},\mu\in\R$}.
A short computation implies with
$\varrho_*=\varrho(\pi_*)$ that {$\varphi(\varrho_*)$ is constant in each component of the phases, hence {$\varrho_*$} is as well, as {$\varphi$} is strictly increasing, and then also {$\pi_*$}  has this property, as {$\varrho$} is strictly increasing, by assumption. Furthermore, we obtain in both cases {$[\![\pi]\!]=\sigma H_{\Gamma_*}$}. In addition, in case {\bf (ii)} we also get {$[\![\varphi(\varrho_*)]\!]=0$}.

\medskip

\noindent
Consequently, in both cases the {critical points} of the available energy functional with the proper constraints are the 
{equilibria} of the system.

\bigskip

\noindent
{\bf 3.3 The variational approach: second variation}

\noindent
Next we look at the second variation of the functional
$${\sf C}:= 
{\sf E}_a + \sum_{ij} \mu_{ij}{\sf M}_{ij}, 
\quad \mbox{resp.}\quad {\sf C}:= {\sf E}_a +  \mu{\sf M}.$$
 Another computation yields with {$\varrho_*^\prime= \varrho^\prime(\pi_*)$} the following identity.
\begin{equation}
{\bf (S)}\quad  \langle {\sf C}^{\prime\prime}(e_*)(v,h)|(v,h)\rangle =\int_\Omega \frac{\varrho^\prime_*}{\varrho_*} |v|^2 \,dx +\sigma \int_{\Sigma}\cA_\Sigma h \bar h \,d\Sigma.
\end{equation}
Here
{$\cA_\Sigma = -H_{\Gamma}^\prime$} means the curvature operator on the equilibrium 
hypersurface {$\Sigma=\Gamma_*$}. For a critical point {$e_*$} of
{${\sf E}_a$} with the given constraints to be a minimum, it is necessary that this form is nonnegative on the kernel
of the derivative of the constraints at {$e_*$}.

 In case {\bf (ii)} we have
$$ (v,h)\in{\sf N}({\sf M}^\prime(e_*)) \quad \Leftrightarrow \quad \int_\Omega \varrho^\prime_*v\,dx = [\![\varrho_*]\!]\int_{\Sigma} h \,d\Sigma.$$
This further implies that the equilibrium interface {$\Gamma_*$} is connected, and that the {stability condition}
$$ {\bf (SCii)}\quad\zeta_*:=\frac{(n-1)\sigma}{[\![\varrho(\pi_*)]\!]^2 R_*^2 |\Gamma_*|} \int_\Omega \varrho^\prime(\pi_*)\varrho(\pi_*)\,dx \leq 1$$
holds true.
Note that this number is dimensionless. In fact, if $\Gamma_*=:\Sigma$ is not connected and has, say, $m>1$ components $\Sigma_k$, set $v=0$ and $h=h_k$ constant on $\Sigma_k$ with $\sum_k h_k=0$. Then $(v,h)\in {\sf N}({\sf M}^\prime(e_*))$ and
$$\langle {\sf C}^{\prime\prime}(e_*)(v,h)|(v,h)\rangle = -\frac{\sigma(n-1)|\Sigma|}{mR_*^2} \sum_k h_k^2 <0, \quad\mbox{for } h\neq0,$$
hence ${\sf C}^{\prime\prime}(e_*)$ is not positive semi-definite on ${\sf N}({\sf M}^\prime(e_*))$. On the other hand, if $\Gamma_*$ is connected,
set $v=\varrho_*w$ with $w$ constant on $\Omega$, and $h$ constant on $\Gamma_*$. In this case $(v,h)\in {\sf N}({\sf M}^\prime(e_*))$ if
$$ \big(\int_\Omega \varrho^\prime_*\varrho_*\,dx\big) w = [\![\varrho_*]\!] |\Gamma_*| h,$$
and
$$\langle {\sf C}^{\prime\prime}(e_*)(v,h)|(v,h)\rangle= \big(\int_\Omega \varrho^\prime_*\varrho_*\,dx\big) w^2 -\frac{\sigma(n-1)|\Gamma_*|}{R^2_*} h^2,$$ is nonnegative if and only if the stability condition {\bf (SCii)} is valid.

Summarizing, we have

\begin{thm} The Verigin problem with phase transition has the following properties.
\begin{enumerate}
\item The {total mass} is preserved along smooth solutions.
\item The {available energy} is a strict Lyapunov functional.
\item The non-degenerate {equilibria} consist of  constant pressures in the phases and the interface {$\Gamma_*$} is a finite disjoint union of spheres of {\bf common radius} {$R_*>0$}, and {$\Omega_2$} is connected.
\item The {equilibria} are precisely the critical points
of the available energy functional with prescribed total mass.
\item {Onset of Ostwald ripening}:
if the available energy functional with prescribed mass  has a {local minimum} at {$e_*=(\pi_*,\Gamma_*)$}
then {$\Gamma_*$} is connected, and the stability condition {${\bf (SCii)}$} holds.
\item If either {$\Gamma_*$} is disconnected or {$\zeta_*>1$}, then {$e_*$} is a saddle point of {${\sf E}_a$} with constraint
{${\sf M} = {\sf M}^0$}.
\end{enumerate}
In particular, the Verigin problem with phase transition is thermodynamically consistent, and an equilibrium is thermodynamically stable if and only if\, $\Gamma_*$ is connected and  the stability condition  {\bf (SCii)} holds.
\end{thm}

\medskip

\noindent
The case {\bf (i)} without phase transition is more involved. Take any component {$\Omega_{ij}$} of {$\Omega_i$}, 
{$i=1,2$} and $j=1,\ldots,m$. Then we obtain
$$(v,h)\in {\sf N}({\sf M}_{ij}^\prime (e_*)) \quad \Leftrightarrow\quad \int_{\Omega_{ij}}\varrho^\prime_* v \,dx + (-1)^{i+1}\int _{\partial\Omega_{ij}}\varrho_*h \,d\Sigma =0.$$
Decomposing
$$v= v_0 +\sum_{ij} v_{ij} \chi_{\Omega_{ij}},\quad \int_{\Omega_{ij}} v_{0}\,dx =0,\; \mbox{ for all } i,j,$$
and
$$ h = h_0 +\sum_{k=1}^m h_k \chi_{\Sigma_{k}},\quad \int_{\Sigma_{k}} h_0 \,d\Sigma =0,\quad k=1,\ldots,m,$$
where $v_{ij}, h_k$ are constants, and observing
$$\int_\Omega (\varrho^\prime_*/\varrho_*) |v_0|^2\,dx \geq0,\quad (\cA_\Sigma h_0|h_0)_\Sigma \geq0, $$
we see that the form {$\langle{\sf C}^{\prime\prime}(v,h)|(v,h)\rangle$}
is nonnegative on {$\cap_{ij}{\sf N}({\sf M}_{ij}^\prime(e_*))$} if and only if the {stability condition}
$${\bf (SCi)}\quad C_* \quad \mbox{is positive semi-definite on}\quad \R^m $$ is valid.
Here the real symmetric matrix {$C_*$} is defined via its entries
$$ c^*_{kl} = \sum_{ij} \frac{\varrho_{ij}(\pi_*)}{\varrho^\prime_{ij}(\pi_*)|\Omega_{ij}|} \delta^k_{ij}\delta^l_{ij} - \frac{\sigma(n-1)}{R_k^2|\Sigma_k|} \delta_{kl},$$
with
{$ \delta_{ij}^k =1 \;\Leftrightarrow \; \Sigma_k\subset \partial\Omega_{ij},\; \delta^k_{ij}=0\; \mbox{otherwise}.$}
In fact, by the constraints
$$(\varrho^\prime_{ij}(\pi_*)/\varrho_{ij}(\pi_*))|\Omega_{ij}|v_{ij} = (-1)^i \sum_k |\Sigma_k| h_k \delta_{ij}^k,$$
hence, with $(v_0,h_0)=(0,0)$ we have
\begin{align*}
\langle C^{\prime\prime}(e_*)(v,h)|(v,h)\rangle &= \sum_{ij} (\varrho^\prime_{ij}(\pi_*)/\varrho_{ij}(\pi_*))|\Omega_{ij}| |v_{ij}|^2 -\sigma(n-1) \sum_k |\Sigma_k| |h_k|^2/R_k^2\\
&= \sum_{k,l} c_{kl}^* |\Sigma_k| h_k |\Sigma_l|h_l = (C_* \tilde{h}|\tilde{h})
\end{align*}
with $\tilde{h}_k = |\Sigma_k| h_k$.

Summarizing, in case {\bf (i)} we have the following result.

\begin{thm} The Verigin problem without phase transition has the following properties.
\begin{enumerate}
\item The {masses of the components of the phases} are preserved along smooth solutions.
\item The { available energy} is a strict Lyapunov functional.
\item The non-degenerate {equilibria} consist of  constant pressures in the components of the phases and the interface {$\Gamma_*$} is a finite disjoint union of spheres of {\bf arbitrary radii}.
\item The {equilibria} are precisely the critical points
of the available energy functional with prescribed total masses of the components of the phases.
\item
If the available energy functional with prescribed masses  has a {local minimum} at {$e_*=(\pi_*,\Gamma_*)$}
then {${\bf (SCi)}$} holds.
\item If {${\bf (SCi)}$} does not hold, then {$e_*$} is a saddle point of {${\sf E}_a$} with the constraints
{${\sf M}_{ij} = {\sf M}_{ij}^0$}.
\end{enumerate}
In particular, the Verigin problem without phase transition is thermodynamically consistent, and an equilibrium is thermodynamically stable if and only if the stability condition  {\bf (SCi)} holds.
\end{thm}
\goodbreak

\section{ Local Well-Posedness of the Verigin Problems}
To prove local well-posedness we investigate the principal part of the linearization of problem \eqref{Ve-without} and \eqref{Ve-with},
respectively.
Here we follow the same steps as in~\cite[Section 1.3.2]{PrSi16}:
we choose a smooth reference manifold $\Sigma\subset\Omega$ which is close to $\Gamma_0$ and represent
the moving surface $\Gamma(t)$ as a graph in normal direction of $\Sigma$, parameterized by a height function 
$h(t,\cdot)$, that is, we write 
$$\Gamma(t)=\{ p+ h(t,p)\nu_\Sigma(p): p\in\Sigma,\; t\geq0\},$$
at least for small $|h|_\infty$.
This yields a diffeomorphism from $\Sigma$ onto $\Gamma(t)$ which will then be extended
to all of $\bar{\Omega}$ by means of the Hanzawa-transform
$$ \Xi_h(t,x) = x +\chi(d_\Sigma(x)/a)h(t,\Pi_\Sigma(x))\nu_\Sigma(\Pi_\Sigma(x))=:x+\xi_h(t,x).$$
Here $\chi$ denotes a suitable cut-off function. More precisely, $\chi\in\cD(\R)$,
$0\leq\chi\leq 1$, $\chi(r)=1$ for $|r|<1/3$, and $\chi(r)=0$ for $|r|>2/3$.
With the help of this transformation, the equations in~\eqref{Ve-without} and~\eqref{Ve-with}  can be expressed with respect to
the variables $(v, h)$, where $v$ stands for the transformed pressure, and $h$ denotes the hight function
introduced above. 
Once the transformed system is obtained, one can derive the linearization at an initial value $(v_0,h_0)$.
In order to keep this manuscript at a reasonable length, we refrain from giving details, and instead refer
to the monograph~\cite{PrSi16} where the technical steps are explained.

\bigskip

\noindent
{\bf 4.1 The principal linearization}

\noindent
{\bf (a)}\, In the bulk {$\Omega\setminus\Sigma$}:
\begin{equation*}
\varrho^\prime_0(x)\partial_t v  - \varrho_0(x){\rm div}(a(x)\nabla v)=\varrho_0^\prime(x) f_v,\\
\end{equation*}
where $a(x) = k_0(x)I+ 2k_1(x)\nabla v_0\otimes\nabla v_0$, with the abbreviations
$\varrho_0(x)=\varrho(v_0(x))$, $\varrho_0^\prime(x)= \varrho^\prime(v_0(x))$, $k_0(x)= k(v_0(x), |\nabla v_0(x)|^2)$, $k_1(x) =\partial_2k(v_0(x),|\nabla v_0(x)|^2)$.
\medskip\\
{\bf (b)}\, On the interface {$\Sigma$}: 
\begin{equation*}
[\![v]\!]-\sigma \Delta_\Sigma h = g_h,
\end{equation*}
and in the case without phase transition
\begin{equation*}
\begin{aligned}
-[\![\nu(x)\cdot a(x)\nabla v]\!] = g_v,\\
\partial_th + \nu(x)\cdot a(x) \nabla  v = f_h.
\end{aligned}
\end{equation*}
If phase transition is present, we have instead
\begin{equation*}
\begin{aligned}
\mbox{}[\![ v/\varrho_0(x)]\!]= g_v,\\
\mbox{}[\![\varrho_0(x)]\!]\partial_t h +[\![\varrho_0(x) \nu(x)\cdot a(x) \nabla v]\!] = f_h.
\end{aligned}
\end{equation*}
\noindent
{\bf (c)}\, On the outer boundary {$\partial\Omega$}:
{$ \partial_\nu v=0.$}\\
{\bf (d)}\, Initial conditions:
{$h(0)=h_0,\quad v(0)=v_0.$}

\bigskip
\goodbreak
\noindent
{\bf 4.2 The principal symbols}

\noindent
In the interior, the problem is clearly parabolic, due to the assumptions
$$\varrho(p),\varrho^\prime(p),k(p,s),k(p,s)+2s\partial_2k(p,s)>0,\quad p\in\R,\; s\ge 0.$$
So we only have to look at the interface {$\Sigma$}. Freezing coefficients, flattening the interface and solving the bulk problems,
this yields the following {boundary symbols}, where {$\lambda$} denotes the covariable of time, and {$\xi$} that of the tangential space directions. We set
$$ {\sf n}_i(\lambda,\xi) = ( (\varrho^\prime_i \lambda/\varrho_i +  a_i(\xi,\xi))a_i(\nu,\nu)-a(\xi,\nu)^2)^{1/2}\quad i=1,2.$$
This is the symbol of a {\em parabolic Dirichlet-to-Neumann operator}. From  Pr\"uss and Simonett \cite{PrSi16}, Sections 6.6 and 6.7,
we obtain the boundary symbols of the linearized Verigin problems.

\medskip

\noindent
{\bf (i)}{\em  Without Phase Transition.}\\
In this case the boundary symbol becomes
$$ s(\lambda,\xi)=\lambda + \frac{ {\sf n}_1(\lambda,\xi)  {\sf n}_2(\lambda,\xi)}{{\sf n}_1(\lambda,\xi)+ {\sf n}_2(\lambda,\xi)}\sigma|\xi|^2.$$
For this case we refer to Pr\"uss and Simonett \cite{PrSi16}, Section 6.7.1.

\medskip

\noindent
{\bf (ii)} {\em With Phase Transition.}\\
By a similar computation as in Pr\"uss and Simonett \cite{PrSi16}, Section 6.6.3,  we have for the boundary symbol
$$ s(\lambda,\xi)=[\![\varrho]\!]^2\lambda + \Big( \varrho_1^2{\sf n}_1(\lambda,\xi) +  \varrho_2^2{\sf n}_2(\lambda,\xi)\Big)\sigma|\xi|^2.$$
Observe that both symbols are equivalent to the boundary symbol of {the standard Stefan problem with surface tension}, namely they are equivalent to the symbol
$$ s_0(\lambda,\xi) = \lambda + |\xi|^2(\lambda+|\xi|^2)^{1/2}.$$
Therefore, the analytical setting, maximal {$L_p$}-regularity, and also the local existence proof are the same
as for the Stefan problem with surface tension!
\\
So the spaces for $(v,h)$ are
$$ v \in H^1_{p,\mu}(J;L_p(\Omega))\cap L_{p,\mu}(J;H^2_p(\Omega\setminus \Sigma)),$$
$$h\in W^{3/2-1/2p}_{p,\mu}(J;L_p(\Sigma))\cap W^{1-1/2p}_{p,\mu}(J;H^2_p(\Sigma))\cap L_{p,\mu}(J; W^{4-1/p}_p(\Sigma)).$$
Here {$\mu\in (1/p,1]$} indicates a time weight, cf.\ Pr\"uss-Simonett \cite{PrSi16}.

\bigskip
\goodbreak
\noindent
{\bf 4.3 Local well-posedness}

\noindent
We rewrite the Hanzawa-transformed problem as
$$ Lz = N(z),$$
where {$z=(v,h)$} collects the system variables.

Define the space of solutions {$\EE(a)$} on the time interval {$J=[0,a]$} by means of
{\begin{eqnarray*}
&& v\in H^1_p(J;L_p(\Omega;\R))\cap L_p(J;H^2_p(\Omega\setminus\Sigma;\R)=:\Red{\EE_{v}(a)},\\
&& h\in W^{3/2-1/2p}_p(J;L_p(\Sigma))\cap W^{1-1/2p}_p(J;H^{}_p(\Sigma))
\cap L_p(J;W^{4-1/p}_p(\Sigma))=:\EE_h(a),\\
&&{\EE}(a):=\{(v,h)\in\EE_v(a)\times\EE_h(a): (v,h) \text{ satsify the compatibility conditions}\}.
\end{eqnarray*}}
From maximal regularity we obtain that
{$L:\EE(a)\to\FF(a)$} is an {isomorphism}, and {$N:\EE(a)\to \FF(a)$} is of class {$C^1$}, provided {$p>n+2$}.
We skip here the precise description  of the data space {$\FF(a):=L\EE(a)$}.
Note that the embeddings
{\begin{eqnarray*}
&&\EE_{v}(a)\hookrightarrow C(J;W^{2-2/p}_p(\Omega\setminus\Sigma))\hookrightarrow C(J;BUC^{1+\alpha}(\Omega\setminus\Sigma))^{n+1},\\
&&\EE_h(a)\hookrightarrow C^1(J;W^{2-3/p}_p(\Sigma)) \cap C(J;W^{4-3/p}_p(\Sigma))\\
&&\mbox{}\qquad\; \hookrightarrow C^1(J;C^{1+\alpha}(\Sigma))\cap C(J;C^{3+\alpha-1/p}(\Sigma)),
\end{eqnarray*}}
with {$\alpha=1-(n+2)/p>0$} are valid.
The nonlinearity {$N$} contains

\medskip

\noindent
\begin{itemize}
\item[*] lower order terms which can be made small by smallness of $a>0$;
\item[*] highest order terms carry $\nabla_\Sigma h$ which are small by smallness of $h_0$.
\end{itemize}
\medskip

\noindent
Therefore, we may apply the {contraction mapping principle} to obtain local well-posedness of the transformed problem
for initial data $(v_0,h_0)\in W^{2-2/p}_p(\Omega\setminus\Sigma)\times W_p^{4-3/p}(\Sigma)$, satisfying appropriate compatibility conditions.
We refer to the monograph Pr\"uss and Simonett  \cite{PrSi16}, Chapter 9, for more details.

\section{Stability of equilibria}
For stability of the equilibria we have to study the spectrum of the the linearization of the problems. 
We observe that these spectra only consist of a sequence of eigenvalues of finite multiplicity converging to infinity, 
due to compact embeddings, as $\Omega$ is bounded.

\goodbreak
\bigskip

\noindent
{\bf 5.1 The eigenvalue problem at an equilibrium}

 \noindent
For the case without phase transition: in the bulk {$\Omega\setminus\Sigma$}, $\Sigma:=\Gamma_*$,:
\begin{equation}\label{linb}
\varrho^\prime_*\lambda v  -\varrho_*k_*\Delta v =0.
\end{equation}
On the interface {$\Sigma:=\Gamma_*$}:
\begin{equation}\label{linin1}
\begin{aligned}
\mbox{}[\![v]\!] +\sigma \cA_\Sigma h &=0,\\
\mbox{}[\![ k_*\partial_\nu v]\!] &=0,\\
\lambda h + k_* \partial_\nu v & =0.
\end{aligned}
\end{equation}
On the outer boundary {$\partial\Omega$}:
{$ \partial_\nu v =0.$}

\medskip

\noindent
Here 
$\varrho_*=\varrho(\pi_*)$, $\varrho^\prime_*=\varrho^\prime(\pi_*)$,  $k_*=k(\pi_*,0)$,
and $\cA_\Sigma=-(n-1)/R_*^2-\Delta_\Sigma$ is the linearization of the curvature.
If phase transition is present, the interface conditions have to be replaced by
\begin{equation}\label{linin2}
\begin{aligned}
\mbox{}[\![v]\!] +\sigma \cA_\Sigma h &=0,\\
\mbox{}[\![v/\varrho_*]\!]&=0,\\
\mbox{}[\![\varrho_*]\!]\lambda h + [\![ \varrho_*k_*\partial_\nu v]\!]&=0.
\end{aligned}
\end{equation}


\noindent
In both cases, taking the {$L_2$}-inner product of \eqref{linb} with {$v/\varrho_*$} leads to
$$ \lambda\big[\int_\Omega \frac{\varrho^\prime_*}{\varrho_*}|v|^2\,dx +\sigma \int_\Sigma \cA_\Sigma h \bar h \,d\Sigma\big] 
+ \int_\Omega k_* |\nabla v|^2\,dx =0,$$
for any eigenvalue {$\lambda\in \C$} and eigenvector {$(v,h)$}. Therefore, all eigenvalues are real, and there are 
{no positive eigenvalues} if and only if
$$ \int_\Omega \frac{\varrho^\prime_*}{\varrho_*}|v|^2\,dx +\sigma \int_\Sigma \cA_\Sigma h \bar h \,d\Sigma\geq0,$$
for all relevant {$(v,h)\neq0$}.
Take any component {$\Omega_{ij}$} of {$\Omega_i$}, and integrate \eqref{linb} over $\Omega_{ij}$. Then we obtain in the first case for {$\lambda\neq0$}
$$\int_{\Omega_{ij}}\varrho^\prime_* v \,dx + (-1)^{i+1}\int _{\partial\Omega_{ij}}\varrho_*h \,d\Sigma =0.$$
This resembles the constraints in case {\bf (i)} found in Section 2.3.
Integrating \eqref{linb} over $\Omega$, in the second case we find
$$ \int_\Omega \varrho^\prime_* v \,dx =[\![\varrho_*]\!]\int_\Sigma h \,d\Sigma,$$
in accordance with the variational approach in case {\bf (ii)}.

\medskip

\noindent
Hence, we may conclude that there are no nontrivial eigenvalues with negative real parts, provided the equilibrium is {thermodynamically stable},
in the sense that {$C_*$} is positive semi-definite in the first case, and {$\Gamma_*$} is connected and {$\zeta_*\leq 1$} in the second case.

It is not difficult to show that the kernel of the linearization {$L$} equals the tangent space of {$\cE$} at an equilibrium {$e_*\in \cE$}. Moreover, we can prove that $0$ is a semi-simple eigenvalue of {$L$}, if and only if {${\rm det}\, C_*\neq0$} in the first case, and {$\zeta_*\neq 1$} in the second case.

Assuming the latter, we can also show that in case {\bf (i)} the number of  negative eigenvalues of {$L$} equals the number of negative eigenvalues of  {$C_*$}, and that in case
{\bf (ii)} we have $m$ positive eigenvalues if {$\zeta_*> 1$}, otherwise {$m-1$}.

These assertions will be proved in the following subsections.
\goodbreak
\bigskip

\noindent
{\bf 5.2 The kernel of the linearization}

\noindent
For the case {\bf (ii)} with phase transition we introduce the linearization operator $L_2$ in $X_0= L_p(\Omega)\times W^{2-2/p}_p(\Sigma)$ by means of
$$ L_2(v,h) = \Big(-\frac{\varrho_*k_*}{\varrho^\prime_*} \Delta v, \frac{[\![ \varrho_*k_*\partial_\nu v]\!]}{[\![\varrho_*]\!]}\Big), \quad (v,h)\in {\sf D}(L),$$
where, with $X_1 = H^2_p(\Omega\setminus\Sigma)\times W^{2-2/p}_p(\Sigma)$, the domain of $L_2$ is given by
$${\sf D}(L_2) =\{ (v,h)\in X_1:\, \partial_\nu v=0 \mbox{ on } \partial\Omega,\; [\![v/\varrho_*]\!]=0,\; [\![v]\!] +\sigma \cA_\Sigma h =0 \mbox{ on } \Sigma\}.$$
This operator is the negative generator of a compact analytic $C_0$-semigroup in $X_0$, see Section 4 and Chapter 6 in Pr\"uss and Simonett \cite{PrSi16}. Therefore, its spectrum consists only of discrete eigenvalues of finite algebraic multiplicity, clustering at infinity.

\medskip

\noindent
{\bf (a)} To compute the kernel ${\sf N}(L_2)$, suppose $L_2(v,h)=0$. Multiplying the equation for $v$ with $\varrho^\prime_* v/\varrho_*$, employing the boundary and interface conditions, we obtain
\begin{align*}
 0&= -\int_\Omega k_*\Delta v v\,dx = \int_\Omega k_*|\nabla v|^2 \,dx + \int_\Sigma [\![k_*\partial_\nu v v ]\!] \,d\Sigma\\
 &= \int_\Omega k_*|\nabla v|^2 \,dx + \int_\Sigma [\![\varrho_*k_*\partial_\nu v  ]\!]v/\varrho_* \,d\Sigma = \int_\Omega k_*|\nabla v|^2 \,dx.
\end{align*}
This implies that $v$ is constant in the components of the phases, and by the interface condition $[\![v/\varrho_*]\!]=0$ we get $v=\alpha_0 \varrho_*$, for some constant $\alpha_0$. Employing the interface condition $[\![v]\!]+\sigma \cA_\Sigma h=0$ this yields
$$ h = \alpha_0 \gamma_* + \sum_{k=1}^m\sum_{i=1}^n \alpha_{ik} Y_{i}^k,
\quad \gamma_* =\frac{[\![\varrho_*]\!]R^2_*}{\sigma(n-1)} ,$$
for some constants $\alpha_{ik}$, where $Y_{i}^k$ denote the spherical harmonics of degree one for the components $\Sigma_k$ of $\Sigma$. Therefore, the kernel of $L_2$ has dimension $(nm+1)$, and ${\sf N}(L_2)$ equals the tangent space $T_{e_*}\cE$ at the equilibrium $e_*$.

\medskip

\noindent
{\bf (b)} Next we show that the eigenvalue $0$ is semi-simple for $L_2$. So let us assume that $L_2^2(w,k)=(0,0)$. Then
$$L_2(w,k)= \alpha_0(\varrho_*,\gamma_*) + \sum_{ik} \alpha_{ik}(0, Y_i^k),$$
for some constants $\alpha_0,\alpha_{ik}$. Integrating the equation for $w$ over $\Omega$ with weight $\varrho^\prime_*$, this yields
\begin{align*}
\alpha_0 (\varrho^\prime_*|\varrho_*)_\Omega &= -\int_\Omega \varrho_* k_* \Delta w dx 
= \int_\Sigma [\![\varrho_*k_* \partial_\nu w]\!]\,\,d\Sigma\\
&= [\![\varrho_*]\!]\int_\Sigma (\alpha_0 \gamma_* +\sum_{ik} \alpha_{ik} Y_i^k) \,d\Sigma 
= [\![\varrho_*]\!]\alpha_0\gamma_*|\Sigma|,
\end{align*}
hence $\alpha_0\neq 0$ is possible if and only if in the stability condition {\bf (SCii)} equality holds, i.e.,\ $\zeta_*=1$. Assuming on the contrary that this is not the case, we obtain $\alpha_0=0$, and then by the equation for $k$ we have
$$ \sum_{ik} \alpha_{ik} Y_i^k =0,$$
which implies $\alpha_{ik}=0$ as the functions $Y_i^k$ are linearly independent. This shows $(w,h)\in {\sf N}(L)$, i.e.,\ $0$ is a semi-simple eigenvalue of $L_2$ if and only if $\zeta_*\neq1$. Otherwise, the algebraic multiplicity raises by $1$.

\bigskip

\noindent
Next we consider the case {\bf (i)} without phase transition. Here we have
$$ L_1(v,h) = (-\frac{\varrho_*k_*}{\varrho^\prime_*} \Delta v, k_*\partial_\nu v), \quad (v,h)\in {\sf D}(L_1),$$
with domain
$${\sf D}(L_1) =\{ (v,h)\in X_1:\, \partial_\nu v=0 \mbox{ on } \partial\Omega,\; [\![k_*\partial_\nu v]\!]=0,\,[\![v]\!] +\sigma \cA_\Sigma h =0 \mbox{ on } \Sigma\}.$$
This operator is also the negative generator of a compact analytic $C_0$-semigroup in $X_0$. Therefore, its spectrum consists only of discrete eigenvalues of finite algebraic multiplicity, clustering at infinity.

\medskip

\noindent
{\bf (a)} To compute the kernel ${\sf N}(L_1)$, suppose $L_1(v,h)=0$. Multiplying the equation for $v$ with $\varrho^\prime_* v/\varrho_*$, employing the boundary and interface conditions, we obtain
\begin{align*}
 0&= -\int_\Omega k_*\Delta v \,dx = \int_\Omega k_*|\nabla v|^2 \,dx + \int_\Sigma [\![k_*\partial_\nu v v ]\!] \,d\Sigma = \int_\Omega k_*|\nabla v|^2 \,dx.
\end{align*}
This implies that $v=v_{ij}$ is constant in the components $\Omega_{ij}$ of the phases. Employing the interface condition $[\![v]\!]+\sigma \cA_\Sigma=0$ this yields
$$ \sum_{ij} (-1)^{i}\delta^k_{ij}v_{ij} +\sigma\cA_\Sigma h=0\quad\mbox{on}\;\; \Sigma_k,\quad 1\le k\le m,$$
which implies
$$ h = \sum_{k=1}^m h_k \chi_{\Sigma_k} + \sum_{ik} \alpha_{ik} Y_i^k, \quad h_k =  \frac{R_k^2}{\sigma(n-1)} \sum_{ij} (-1)^{i}\delta^k_{ij}v_{ij}.$$
Thus the dimension of the kernel ${\sf N}(L_1)$ equals $(mn+m+1)$, and the tangent space $T_{e_*}\cE$  equals ${\sf N}(L_1)$.

\medskip

\noindent
{\bf (b)} Next we show that eigenvalue the $0$ is semi-simple for $L_1$. Let us assume that $L_1^2(w,k)=(0,0)$. Then
$$L_1(w,k)= (\sum_{ij}(v_{ij}\chi_{ij},\sum_k h_k\chi_{\Sigma_k}) + \sum_{lk} \alpha_{lk} (0,Y_l^k),$$
for some constants $v_{ij}, \alpha_{lk}$, and $h_k$ as defined above,
and $\chi_{ij}:=\chi_{\Omega_{ij}}$. Integrating the equation for $w$ over $\Omega_{ij}$ this yields
\begin{equation*}
\begin{aligned}
|\Omega_{ij}| v_{ij} &= -(\varrho_{ij}(\pi_*)/\varrho^\prime_{ij}(\pi_*))\int _{\Omega_{ij}} k_* \Delta w \,dx =  
(\varrho_{ij}(\pi_*)/\varrho^\prime_{ij}(\pi_*))(-1)^i\int _{\partial\Omega_{ij}} k_* \partial_\nu w \,dx\\
& = (\varrho_{ij}(\pi_*)/\varrho^\prime_{ij}(\pi_*))(-1)^i \big[\sum_l h_l\int_{\partial\Omega_{ij}}\chi_{\Sigma_k}\,d\Sigma + \sum_{lk} \alpha_{lk} \int_{\partial\Omega_{ij}}Y_l^k \,d\Sigma\big]\\
&= (\varrho_{ij}(\pi_*)/\varrho^\prime_{ij}(\pi_*))(-1)^i\sum_l \delta^l_{ij}|\Sigma_l| h_l.
\end{aligned}
\end{equation*}
Dividing by $|\Omega_{ij}|$ and summing over $i,j$, we obtain
\begin{align*}
 \frac{\sigma(n-1)}{R^2_k} h_k &=\sum_{ij} (-1)^{i}\delta^k_{ij}v_{ij}
 = \sum_l\sum_{ij}\frac{\varrho_{ij}(\pi_*)}{\varrho^\prime_{ij}(\pi_*)|\Omega_{ij}|} \delta^k_{ij}\delta^l_{ij}|\Sigma_l|h_l.
 \end{align*}
This implies that the vector $\tilde{h}$ with components $\tilde{h}_k =|\Sigma_k|h_k$ is an eigenvector of $C_*$. 
So if ${\rm det}\, C_*\neq0$ this yields $h_k=$ for all $k$; hence $v$ is constant all over $\Omega$, and  so integrating once more the equation for $w$ with weight $\varrho^\prime_*/\varrho_*$ we obtain also $v=0$. This shows that $0$ is a semi-simple eigenvalue of $L_1$ if and only if $C_*$ is invertible; otherwise the algebraic multiplicity of $0$ raises by ${\rm dim}\,{\sf N}(C_*)$.

\bigskip

\noindent
{\bf 5.3 Normal stability and normal hyperbolicity.}

\noindent
We begin with case {\bf (ii)} where phase transition is present, following the ideas in our monograph \cite{PrSi16}, Chapter 10.

\medskip

\noindent
{\bf (a)} Consider the elliptic problem
\begin{equation}\label{evpii}
\begin{aligned}
\varrho^\prime_*\lambda v  -\varrho_*k_*\Delta v &=0 &&\mbox{in}\;\; \Omega\setminus\Sigma,\\
\partial_\nu v &=0 && \mbox{on}\;\; \partial\Omega, \\
\mbox{}[\![v/\varrho_*]\!]&=0 && \mbox{on}\;\; \Sigma,\\
- [\![ \varrho_*k_*\partial_\nu v]\!]&=g && \mbox{on}\;\; \Sigma.
\end{aligned}
\end{equation}
Given $g\in H^{1/2}_2(\Sigma)$, by elliptic theory, this problem has a unique solution  $v\in H^2_p(\Omega\setminus\Sigma)$, 
for each $\lambda>0$. We then set
$$ [\![\varrho_*]\!]T_\lambda g: = [\![v]\!] = [\![\varrho_*v/\varrho_*]\!] = [\![\varrho_*]\!] v/\varrho_*.$$
This simplifies the eigenvalue problem considerably. In fact, $\lambda>0$ is an eigenvalue of the linearization $L_2$ 
at equilibrium $e_*$ if and only if $0$ is an eigenvalue of
$$ B_\lambda = [\![\varrho]\!]^2 \lambda T_\lambda +\sigma \cA_\Sigma.$$
Next, multiplying \eqref{evpii} with $v/\varrho_*$ and integrating by parts we obtain the important identity
$$ \lambda \int_\Omega (\varrho^\prime_*/\varrho_*) |v|^2 \,dx + \int_\Omega k_*|\nabla v|^2 \,dx = (T_\lambda g|g)_\Sigma.$$
Hence $T_\lambda$ is positive semi-definite on $L_2(\Sigma)$. In a similar way one can show that $T_\lambda$ is selfadjoint, and it is compact in $L_2(\Sigma)$,
as $T_\lambda\in \cB( H^{1/2}_2(\Sigma);H^{3/2}_2(\Sigma)$. Therefore, $B_\lambda$ is selfadjoint with compact resolvent, hence its spectrum consists only of semi-simple real eigenvalues.

\medskip

\noindent
{\bf (b)}  We need to compute the limit of $\lambda T_\lambda$ as $\lambda\to 0$. For this purpose we introduce first the 
bulk operator $A_2$ in $L_2(\Omega)$ by means of
$$ A_2 v = - \frac{\varrho_* k_*}{\varrho_*^\prime} \Delta v, \quad v\in {\sf D}(A_2),$$
with domain
$${\sf D}(A_2)= \{ v\in H^2_2(\Omega\setminus \Sigma):\; \partial_\nu v=0 \mbox{ on } \partial\Omega,\; 
[\![v/\varrho_*]\!]=0,\, [\![\varrho_*k_* \partial_\nu v]\!]=0 \mbox{ on } \Sigma\}.$$
This operator is selfadjoint and positive semi-definite w.r.t.\, the inner product
$$\langle v_1|v_2\rangle := \int_\Omega v_1\overline{v_2} \varrho^\prime_* \,dx/\varrho_*,$$
and by compact embedding has compact resolvent.
We decompose the solution $v$ of \eqref{evpii} as $v = v_0+v_2$, where $v_0$ solves \eqref{evpii}, with $g=h$, for a fixed $\lambda_0>0$. Then $v_2$ solves the problem
$$\lambda v_2+A v_2= (\lambda_0-\lambda)v_0,\quad \mbox{ hence } v_2=(\lambda_0-\lambda)(\lambda+A_2)^{-1} v_0.$$
Let $P_0$ denote the orthogonal projection onto ${\sf N}(A_2)$. Then it is well-known that $\lambda(\lambda+A_2)^{-1} \to P_0$ as $\lambda\to0.$
Therefore, we obtain
$$\lambda v = \lambda v_0 +(\lambda_0-\lambda)\lambda(\lambda +A_2)^{-1}v_0 \to \lambda_0P_0v_0,$$
as $\lambda\to0$. It is easy to see that the kernel of $A_2$ is one-dimensional and spanned by the function $\varrho_*$, which is constant in the phases.
This implies that the projection $P_0$ is given by $P_0 = \varrho_*\otimes \varrho_*/(\varrho^\prime_*|\varrho_*)_\Omega.$ 
Hence
\begin{equation*}
\begin{aligned}
(\varrho^\prime_*|\varrho_*)_\Omega\lambda_0 P_0 v_0 &= \varrho_* \langle \lambda_0 v_0|\varrho_*\rangle = \varrho_*\int_\Omega \lambda_0\varrho^\prime_* v_0\,dx \\
&= \varrho_*\int_\Omega \varrho_*k_* \Delta v_0 \,dx= - \varrho_*\int_\Sigma [\![\varrho_*k_* \partial_\nu v_0]\!]\,d\Sigma = \varrho_*\int_\Sigma h \,d\Sigma,
\end{aligned}
\end{equation*}
i.e.,\ we have
$$P_0 \lambda_0 v_0= \big(\varrho_*/(\varrho^\prime_*|\varrho_*)_\Omega\big) \int_\Sigma h \,d\Sigma.$$
Taking the jump of $P_0 \lambda_0v_0$ across $\Sigma$ this finally yields
$$ B_0 h = \big([\![\varrho_*]\!]^2/(\varrho^\prime_*|\varrho_*)_\Omega) \int_\Sigma h \,d\Sigma +\sigma \cA_\Sigma h.$$
Decomposing $h= h_0 +\sum_k h_k\chi_{\Sigma_k}$ with constants $h_k$ such that $\int_{\Sigma_k} h_0=0$ for all $k$, and observing that $\cA_\Sigma$ is positive semi-definite on functions with mean zero over each component $\Sigma_k$ of $\Sigma$, we may assume $h_0=0$ in the sequel. If $\sum_k h_k =0$, then
$$ B_0 h = -\frac{\sigma(n-1)|\Sigma|}{m R_*^2} h =:-\mu_0 h,$$
which shows that $-\mu_0$ is an $(m-1)$-fold eigenvalue of $B_0$. Finally consider $h$ constant over $\Sigma$. Then
$$ B_0h =  [\big([\![\varrho_*]\!]^2|\Sigma|/(\varrho^\prime_*|\varrho_*)_\Omega)- \frac{\sigma(n-1)}{ R_*^2}]h =\mu_1 h.$$
This yields another eigenvalue $\mu_1$ of $B_0$ which is negative if the stability condition {\bf (SCii)} does not hold.

\medskip

\noindent
{\bf (c)} Next we show that for large $\lambda$ the operator $B_\lambda$ is positive definite in $L_2(\Sigma)$. Let ${\sf a}_k$ be an orthonormal basis of
${\sf N}(\cA_\Sigma)\oplus {\sf N}( (n-1)/R_*^2 +\cA_\Sigma)$ and let $P= \sum_k a_k\otimes a_k$ denote the corresponding orthogonal projection in $L_2(\Sigma)$. Then with $Q=I-P$, $\cA_\Sigma$ is positive definite on ${\sf R}(Q)={\sf N}(P)$. 
Now we assume the contrary, i.e., there exist sequences $\lambda_n\to\infty$,
$h_n\in L_2(\Sigma)$ with $|h_n|_\Sigma =1$, such that $(B_{\lambda_n}h_n|h_n)_\Sigma \leq 1/n$, for all $n\in\N.$ Then
$$[\![\varrho_*]\!]^2\lambda_n (T_{\lambda_n}h_n|h_n)_\Sigma \leq (B_{\lambda_n}h_n|h_n) -\sigma (\cA_\Sigma Ph_n|Ph_n)_\Sigma \leq C$$
is bounded, hence the corresponding solutions $v_n$ of \eqref{evpii} satisfy 
$$\lambda_n |v_n|_2+ \sqrt{\lambda_n} |\nabla v_n|\leq C.$$
Therefore, $\lambda_nv_n \rightharpoonup w$ weakly in $L_2(\Omega)$
along a subsequence, which will be denoted again by $\lambda_n v_n$. Taking a test function $\phi\in \cD(\Omega\setminus \Sigma)$, this yields
$$ (\varrho^\prime_* \lambda_n v_n|\phi)_\Omega = (\varrho_* k_* \Delta v_n|\phi)_\Omega 
=  (\varrho_* k_* v_n|\Delta\phi)_\Omega \to 0,$$
as $n\to \infty$, hence $w=0$. Next we extend the functions ${\sf a}_k$ from $\Sigma$ to functions 
$a_k\in {_0}H^1_2(\Omega)$. Then
\begin{equation*}
\begin{aligned}
 (h_n|{\sf a}_k)_\Sigma &= -\int_\Sigma [\![ \varrho_*k_* \partial_\nu v_n]\!]{\sf a}_k \,d\Sigma\\
  &= \int_\Omega {\rm div}(\varrho_*k_* \nabla v_n {\sf a}_k) \,dx\\
 &= \int_\Omega \varrho_*^\prime \lambda_n v_n {\sf a}_k\,dx + \int_\Omega \varrho_*k_*\nabla v_n\cdot\nabla{\sf a}_k \,dx\to 0,
 \end{aligned}
 \end{equation*}
 as $n\to \infty$, for each $k$, which shows $Ph_n \to 0$. But as $\cA_\Sigma$ is positive definite on ${\sf N}(P)$, 
 this also yields $Qh_n\to 0$ in $L_2(\Sigma)$, a contradiction to $|h_n|_\Sigma=1$.

\medskip

\noindent
{\bf (d)} We have shown that in case $\Sigma$ consists of $m$ components, $B_0$ has $(m-1)$ negative eigenvalues if the stability condition {\bf (SCii)} holds and $m$ negative eigenvalues otherwise, and $B_\lambda$ has no negative eigenvalues for large $\lambda$. As $\lambda$ runs from zero to infinity these negative eigenvalues have to cross the imaginary axis through $0$, this way inducing an equal number of  positive eigenvalues of $L_2$. This proves the statements in case {\bf (ii)}.

Next we deal with case {\bf (i)} without phase transition. The arguments are similar, and it is enough to carry out steps {\bf (a)} and {\bf (b)}. The remaining steps will be the same as in case {\bf (ii)}, so we may skip them.

\medskip

\noindent
{\bf (a)} Consider the elliptic problem
\begin{equation}\label{evpi}
\begin{aligned}
\varrho^\prime_*\lambda v  -\varrho_*k_*\Delta v &=0  &&\mbox{in}\;\; \Omega\setminus\Sigma,\\
\partial_\nu v &=0 &&\mbox{on}\;\; \partial\Omega, \\
\mbox{}[\![k_* \partial_\nu v]\!]&=0 && \mbox{on}\;\; \Sigma,\\
- k_*\partial_\nu v&=g && \mbox{on} \;\;\Sigma.\
\end{aligned}
\end{equation}
Given $g\in H^{1/2}_2(\Sigma)$, by elliptic theory this problem has a unique solution  $v$, for each $\lambda>0$. Here we  set
$ T_\lambda g: = [\![v]\!]$.
Then $\lambda>0$ is an eigenvalue of the linearization at equilibrium $e_*$ if and only if
$0$ is an eigenvalue of
$$ B_\lambda =  \lambda T_\lambda +\sigma \cA_\Sigma.$$
Next, multiplying \eqref{evpi} with $v/\varrho_*$ and integrating by parts we obtain the identity
$$ \lambda \int_\Omega (\varrho^\prime_*/\varrho_*) |v|^2 \,dx + \int_\Omega k_*|\nabla v|^2 \,dx = (T_\lambda g|g)_\Sigma.$$
Hence $T_\lambda$ is positive semi-definite on $L_2(\Sigma)$. In a similar way one can show that $T_\lambda$ is selfadjoint, and it is compact in $L_2(\Sigma)$,
as $T_\lambda\in \cB( H^{1/2}_2(\Sigma);H^{3/2}_2(\Sigma)$. Therefore, $B_\lambda$ is selfadjoint with compact resolvent, hence its spectrum consists only of semi-simple real eigenvalues.

\medskip

\noindent
{\bf (b)} We proceed in a similar way as in case {\bf(ii)}. Here the operator $A_1$ in $L_2(\Omega)$ is defined by
$$ A_1 v = - \frac{\varrho_* k_*}{\varrho_*^\prime} \Delta v, \quad v\in {\sf D}(A),$$
with domain
$${\sf D}(A_1)= \{ v\in H^2_2(\Omega\setminus \Sigma):\; \partial_\nu v=0 \mbox{ on } \partial\Omega,\;  
[\![k_*\partial_\nu v]\!]=0,\,k_* \partial_\nu v=0 \mbox{ on } \Sigma\}.$$
This operator is selfadjoint and positive semi-definite w.r.t.\, the inner product
$$\langle v_1|v_2\rangle := \int_\Omega v_1\overline{v_2} \varrho^\prime_* \,dx/\varrho_*,$$
and by compact embedding has compact resolvent.  To compute the projection $P_0$, note that the kernel of ${A_1}$ is spanned by the characteristic functions
$\chi_{ij}:=\chi_{\Omega_{ij}}$ of the components of the phases, as any $v\in {\sf N}(A_1)$ is constant on each component of $\Omega\setminus\Sigma$. This yields with $\varrho_{ij}= \varrho_*$ on $\Omega_{ij}$ and similarly for $\varrho^\prime_{ij}$,
$$ P_0 = \sum_{ij} \big(\varrho_{ij}/\varrho_{ij}^\prime|\Omega_{ij}|\big) \chi_{ij}\otimes \chi_{ij}.$$
Next we compute $P_0\lambda_0v_0$ as follows.
\begin{equation*}
\begin{aligned}
P_0\lambda_0 v_0 &=\sum_{ij}  \big(\varrho_{ij}/\varrho_{ij}^\prime|\Omega_{ij}|\big)\chi_{ij}\int_{\Omega_{ij}} \lambda_0\varrho^\prime_*v_0/\varrho_*\,dx\\
&= \sum_{ij}  \big(\varrho_{ij}/\varrho_{ij}^\prime|\Omega_{ij}|\big)\chi_{ij}\int_{\Omega_{ij}} k_*\Delta v_0\,dx\\
&= \sum_{ij}  \big(\varrho_{ij}/\varrho_{ij}^\prime|\Omega_{ij}|\big)\chi_{ij}(-1)^{i+1}\int_{\partial\Omega_{ij}} k_* \partial_\nu v_0 d(\partial\Omega_{ij})\\
&= \sum_{ij}  \big(\varrho_{ij}/\varrho_{ij}^\prime|\Omega_{ij}|\big)\chi_{ij}(-1)^{i}\int_{\partial\Omega_{ij}} h d(\partial\Omega_{ij}),
\end{aligned}
\end{equation*}
hence
$$  P_0\lambda_0 v_0 =\sum_{ij}  \big(\varrho_{ij}/\varrho_{ij}^\prime|\Omega_{ij}|\big)\chi_{ij}(-1)^{i}\sum_l \delta_{ij}^l \int_{\Sigma_l} h \,d\Sigma,$$
where $\delta^k_{ij}= 1$ if $\Sigma_k\subset \partial\Omega_{ij}$ and $\delta^k_{ij}= 0$ otherwise. To compute the jump we note that $[\![\chi_{ij}]\!]=(-1)^{i}$ if $\delta^k_{ij}=1$ and  is 0 otherwise. This yields
$$[\![P_0\lambda_0v_0]\!]=\sum_{kl} \sum_{ij} \big(\varrho_{ij}/\varrho_{ij}^\prime|\Omega_{ij}|\big)\chi_{\Sigma_l} \delta^k_{ij}\delta_{ij}^l\int_{\Sigma_k} h \,d\Sigma,$$
and so decomposing as before $h=h_0 +\sum_k h_k \chi_{\Sigma_k}$, we derive the representation
$$( B_0 h|h)_\Sigma  = \sigma (\cA_\Sigma h_0|h_0)_\Sigma  + \sum_{kl} c^*_{kl} |\Sigma_k|h_k |\Sigma_l| h_l,$$
where the coefficients $c_{kl}^*$ of the matrix $C_*$ have been introduced in Section 3. As a consequence we see that the number of negative eigenvalues of $L_1$
equals the number of negative eigenvalues of $C_*$.

\medskip

\bigskip

\noindent
{\bf 5.4 Nonlinear stability of equilibria}

\noindent
Let {$\cE$} be the set of (non-degenerate) equilibria, and fix some equilibrium {$e_*=(\pi_*,\Gamma_*)\in\cE$}.
Employing the findings from the previous section, we have

\begin{itemize}
\item {$e_*$} is {normally stable} if {$C_*$} is positive definite, resp.\ {$\zeta_*<1$} and {$\Gamma_*$} is connected.
\item {$e_*$} is {normally hyperbolic} if {$C_*$} is indefinite, resp.\ {$\zeta_*>1$} or {$\Gamma_*$} is disconnected.
\end{itemize}

\noindent
Therefore, the {Generalized Principle of Linearized Stability}  due to Pr\"uss, Simonett, Zacher \cite{PSZ09}  yields  our
main result on stability of equilibria.


\goodbreak
\noindent
\begin{thm}
\label{thm:stability}
 Let {$e_*\in\cE$} be a non-degenerate equilibrium such that {${\rm det}\, C_*\neq 0$}, resp.\ {$\zeta_*\neq1$}. 
Then
\begin{itemize}
\item[{\bf (i)}] If {$e_*$}  is normally stable, it is nonlinearly stable, and any solution starting near {$e_*$}
is global and converges to another equilibrium {$e_\infty\in\cE$} at an exponential rate.
\item[{\bf (ii)}] 
If {$e_*$} is normally hyperbolic, then {$e_*$} is nonlinearly unstable. Any solution starting in a neighborhood of {$e_*$}
{and staying near} {$e_*$} exists globally and converges to an equilibrium {$e_\infty\in\cE$} at an exponential rate.
\end{itemize}
\end{thm}
\begin{proof}
The proof parallels that for the Stefan problem with surface tension given in Pr\"uss and Simonett \cite{PrSi16}, Chapter 11.
\end{proof}

\section{Global Behaviour}

In this last section we want to describe the global behaviour of the Verigin problem with and without phase transition.

\bigskip

\noindent
{\bf 6.1 The Local Semiflows}

\noindent
Here we introduce the semiflows induced by the solutions of the problems.
Recall that the closed {$C^2$}-hypersurfaces contained in {$\Omega$}
form a {$C^2$}-manifold, denoted by {$\cMH^2$}.
Charts are obtained via parametrization over a fixed hypersurface, and the
tangent spaces consist of the normal vector fields.\\
As an ambient space for the
state-manifold $\cSM$ of the Verigin problems
we consider the product space {$X_0:= L_p(\Omega)\times \cMH^2$}.
The {compatibility conditions} are given by
\begin{equation}\label{Cvel}
\begin{aligned}
\partial_\nu\pi &=0 && \mbox{on}\;\; \partial\Omega,\\
\mbox{}[\![\pi]\!]&=\sigma H_\Gamma &&\mbox{on}\;\;\Gamma,\\
\quad [\![ k(\pi,|\nabla\pi|^2)\partial_\nu \pi]\!]&=0 &&\mbox{on}\;\;\Gamma,\\
k(\pi,|\nabla\pi|^2)\partial_\nu\pi\in W^{2-6/p}_{p}&(\Gamma) && \mbox{on}\;\;\Gamma,
\end{aligned}
\end{equation}
in the {first case}, while in the {second case} the last two conditions are to be replaced by
\begin{equation}
\label{Cvel-B}
[\![\psi(\varrho)+\varrho\psi^\prime(\varrho)]\!]=0,\quad [\![\varrho(\pi) k(\pi,|\nabla\pi|^2)\partial_\nu\pi]\!]\in W^{2-6/p}_{p}(\Gamma),
\end{equation}
and in this case we additionally require $[\![\varrho]\!]\neq0$.

We define the state manifolds $\cSM$ of the problems as follows
\begin{equation}\label{phasemanif}
\begin{aligned}
\cSM&:=\{(\pi,\Gamma)\in X_0: \pi\in W^{2-2/p}_p(\Omega\setminus\Gamma),
\, \Gamma\in W^{4-3/p}_p,\nonumber\\ 
&\hspace{0.7cm} \mbox{ the compatibility conditions are satisfied} \}.
 \end{aligned}
 \end{equation}
The charts for these manifolds are obtained by the charts induced by those for {$\cMH^2$},
followed by a Hanzawa transformation.
Observe that the compatibility conditions
as well as regularity are preserved by the solutions.

Applying the local existence result and re-parameterizing repeatedly, we obtain
 the {local semiflows} on {$\cSM$}.

\begin{thm}  Let {$p>n+2$} and {$[\![\varrho_0]\!]\neq0$} for the second case.\\
 Then the two-phase Verigin
problems  generate  local semiflows
on their respective state manifolds {$\cSM$}. Each solution {$(\pi, \Gamma)$}
exists on a maximal time interval {$[0,t_+)$}.
\end{thm}

\goodbreak
\bigskip

\noindent
{\bf 6.2. Global existence and asymptotic behaviour}

\noindent
There are a number of obstructions to global existence of the solutions:
\begin{itemize}
\item[-] {\bf regularity}: the norms of either  {$\pi(t)$} or {$\Gamma(t)$} may become unbounded;
\item[-] {\bf geometry}: the topology of the interface may change; or the interface may touch the boundary of {$\Omega$}; 
or a part of the interface may shrink to a point in case {\bf (ii)};
\item[-] {\bf well-posedness}: {$|[\![\varrho(t,x)]\!]|$} may come close to zero in case {\bf (ii)}.
\end{itemize}
\medskip

\noindent
We say that a solution {$(\pi,\Gamma)$} satisfies the {uniform ball condition},
if there is a radius {$r>0$} such that for each {$t\in[0,t_+)$} and
at every point {$p\in\Gamma(t)$} we have
$$\bar{B}(p\pm r\nu_{\Gamma(t)}(p),r)\subset\Omega,\quad\bar{B}(p\pm r\nu_{\Gamma(t)}(p),r)\cap \Gamma(t)=\{p\}.$$
Combining the above results, we obtain the following theorem
on the asymptotic behavior of solutions.

\begin{thm} Let {$p>n+2$}. Suppose that {$(\pi,\Gamma)$}
is a solution of one of the Verigin problems, and assume the following on its maximal interval of existence {$[0,t_+)$}:
\begin{itemize}
\item[{($\alpha$)}]   {$|\pi(t)|_{W^{2-2/p}_p}+|\Gamma(t)|_{W^{4-3/p}_p}\leq M$};
\smallskip
\item[{($\beta$)}]  {$(\pi,\Gamma)$} satisfies the uniform ball condition;
\smallskip 
\item[{($\gamma$)}]  {$c\leq [\![\varrho(t)]\!]$} in case  {\bf (ii)}, for some constant {$c>0$}.
\end{itemize}
Then {$t_+=\infty$}, i.e.,\ the solution exists globally, its limit set {$\omega(\pi,\Gamma)\subset \cE$} is nonempty, and the solution converges in {$\cSM$} to an equilibrium, provided either
\begin{itemize}
\item $\omega(\pi,\Gamma)$ contains a stable equilibrium {$e_\infty$};
\item $(\pi,\Gamma)$ stays eventually near some  {$e_\infty\in\omega(\pi,\Gamma)$}.
\end{itemize}
The converse is also true: if a global solution converges, then ($\alpha$),($\beta$),($\gamma$) are valid.
\end{thm}
\begin{proof}
It can be shown that the closed  $C^2$-hypersurfaces contained in $\Omega$ which bound a region 
$\Omega_1\subset\joinrel\subset\Omega$ 
form a $C^2$-manifold, denoted by $\cMH^2(\Omega)$, see for instance~\cite[Chapter 2]{PrSi16}.
It is also known that each $\Gamma\in\cMH^2(\Omega)$ admits a tubular neighborhood
$$U_a:=\{x\in\R^n: {\rm dist}(x,\Gamma)<a\}$$ 
of width $a=a(\Gamma)>0$ such that the signed distance function
$$d_\Gamma:U_a\to \R,\quad |d_\Gamma(x)|:={\rm dist}(x,\Gamma),$$ 
is well-defined and $d_\Gamma\in C^2(U_a,\R)$, see for instance~\cite[Section 2.3]{PrSi16}. 
Here, by convention,  $d_\Gamma(x)<0$ 
iff $x\in\Omega_1\cap U_a$.
We can then define a {level function} $\varphi_\Gamma$ by means of
\begin{equation*}
\label{level}
\varphi_\Gamma(x):=
\left\{
\begin{aligned}
& d_\Gamma(x)\chi(3d_\Gamma(x)/a)+ {\rm sgn}\,(d_\Gamma(x))(1-\chi(3d_\Gamma(x)/a)),  &&x\in U_a,\\
& \chi_{\Omega_{\rm ex}}(x)-\chi_{\Omega_{\rm in}}(x),  &&x\notin U_a,
\end{aligned}
\right.
\end{equation*}
where $\Omega_{\rm ex}$ and $\Omega_{\rm in}$ denote the exterior  and interior component of 
$\R^n\setminus U_a$,
respectively,
and $\chi$ is a smooth cut-off function with $\chi(s)=1$ if $|s|<1$ and $\chi(s)=0$ 
if $|s|>2$.
The level function $\varphi_\Gamma$ is then of class $C^2$,
$\varphi_\Gamma(x)=d_\Gamma(x)$ for $x\in U_{a/3}$, and
$\varphi_\Gamma(x)=0$ iff $x\in \Gamma$.

Let $\cMH^2(\Omega,r)$ denote the subset of $\cMH^2(\Omega)$ consisting of all $\Gamma\in \cMH^2(\Omega)$ such that $\Gamma\subset \Omega$ satisfies the ball condition with fixed radius $r>0$. This implies in particular that ${\rm dist}(\Gamma,\partial\Omega)\geq 2r$  and all principal curvatures of $\Gamma\in\cMH^2(\Omega,r)$ are bounded by $1/r$. Furthermore, 
the level functions $\varphi_\Gamma $ are well-defined 
for $\Gamma\in\cMH^2(\Omega,r)$ and form a bounded subset of $C^2(\bar{\Omega})$ and the map 
$$\Phi:\cMH^2(\Omega,r)\to C^2(\bar{\Omega}),\quad \Phi(\Gamma)=\varphi_\Gamma,$$ 
is a homeomorphism of the metric space $\cMH^2(\Omega,r)$ onto $\Phi(\cMH^2(\Omega,r))\subset C^2(\bar{\Omega})$, see \cite[Section 2.4.2]{PrSi16}.

Let $s-(n-1)/p>2$. For $\Gamma\in\cMH^2(\Omega,r)$
we define $\Gamma\in W^s_p(\Omega,r)$ if $\varphi_\Gamma\in W^s_p(\Omega)$. In this case the local charts for $\Gamma$ can be chosen of class $W^s_p$ as well. A subset $A\subset W^s_p(\Omega,r)$ is said to be (relatively) compact, if $\Phi(A)\subset W^s_p(\Omega)$ is (relatively) compact. Finally, 
we define
${\rm dist}_{W^s_p}(\Gamma_1,\Gamma_2):= |\varphi_{\Gamma_1}-\varphi_{\Gamma_2}|_{W^s_p(\Omega)}$
for $\Gamma_1,\Gamma_2\in \cMH^2(\Omega,r)$.

Suppose that the assumptions $(\alpha)-(\gamma)$ are valid. 
Then $\Gamma([0,t_+))\subset W^{4-3/p}_p(\Omega,r)$ is bounded, hence relatively compact in
$W^{4-3/p-\ve}_p(\Omega,r)$. Thus $\Gamma([0,t_+))$ can be covered by finitely many balls with centers $\Sigma_k$  such that
$${\rm dist}_{W^{4-3/p-\ve}_p}(\Gamma(t),\Sigma_j)\leq \delta \quad\text{for some $j=j(t)$,\; $t\in[0,t_+)$}.$$ 
Let $J_k=\{t\in[0,t_*):\, j(t)=k\}$. Using for each $k$ a Hanzawa-transformation $\Xi_k$, we see that the pull backs 
$\{\pi(t,\cdot)\circ\Xi_k:\, t\in J_k\}$ are bounded in $W^{2-2/p}_p(\Omega\setminus \Sigma_k)$, 
hence relatively compact in $W^{2-2/p-\ve}_p(\Omega\setminus\Sigma_k)$.
 By well-posedness, we obtain solutions
 $(\pi^1,\Gamma^1)$ with initial configurations $(\pi(t),\Gamma(t))$ in the state manifold $\cSM$ on a common time interval, say $(0,a]$, and by uniqueness we have
 $$(\pi^1(a),\Gamma^1(a))=(\pi(t+a),\Gamma(t+a)).$$ 
 Continuous dependence implies then relative compactness of 
 $$\{(\pi(\cdot),\Gamma(\cdot)):\, 0\leq t<t_+\}$$ in $\cSM;$
in particular $t_+=\infty$ and the orbit $(\pi,\Gamma)(\R_+)\subset\cSM$ is relatively compact.
The available energy is a strict Lyapunov functional, hence the limit set $\omega(\pi,\Gamma)$
of a solution is contained in the set $\cE$ of equilibria.
By compactness, $\omega(\pi,\Gamma)\subset \cSM$ is non-empty, 
hence the solution comes close to $\cE$. Finally, we may apply the convergence result 
Theorem \ref{thm:stability} to complete the sufficiency part of the proof. 
Necessity follows by a compactness argument.
\end{proof}

\end{document}